\documentclass[12pt,reqno]{amsart}
\usepackage{amsmath}
\usepackage[dvips]{graphicx}
\usepackage{amsfonts}
\usepackage{amssymb}
\usepackage{latexsym}
\graphicspath{{Img/}}
\newtheorem{theorem}{Theorem}
\theoremstyle{plain}

\newtheorem{lemma}{Lemma}

\newtheorem{problem}{Problem}

\newtheorem{remark}{Remark}

\numberwithin{equation}{section}

\begin{document}

\title[Do Fermat constants determine Clairaut constants?]{When do Fermat constants completely determine Clairaut constants for branching geodesics on a surface of revolution?}

\author{Anastasios N. Zachos}

\address{University of Patras, Department of Mathematics, GR-26500 Rion, Greece}
 \email{azachos@gmail.com}

\dedicatory{To the geometric vision of Fermat and Clairaut}

\begin{abstract}
We prove that Fermat constants do not completely determine Clairaut constants for three branching geodesics that meet at the weighted Fermat-Torricelli point on a surface of revolution, except the case of a standard sphere in $\mathbb{R}^{3}.$

\end{abstract}\maketitle

\section{Introduction}

We start by considering a surface of revolution $S$ in $\mathbb{R}^{3},$ in a parametric form, which is derived by
rotating the curve $\gamma(u)=(\phi(u)),\psi(u)),$ $x=\phi(u),z=\psi(u)$  lying in the $x,z$ plane about the $z-axis$ (\cite[Chapter~VII,Section~6]{Pogorelov:59}),\cite[p.~212]{Spivak:79}\cite[Section~2.6.1, Example~2.7.3]{VToponogov:05}:

$r(u,v)= (\phi(u)\cos v, \phi(u)\sin(v),\psi(u)).$

Meridians are called the curves of intersection of $S$ with planes passing through the $z-$axis.

Parallels are called the curves of intersection of $S$ with planes perpendicular to the $z-$axis.

We note that meridians and parallels are the curves $v=constant$ and $u=constant$ of $S,$ respectively and they form an orthogonal net, because $F=0.$
Thus, meridians and parallels are lines of curvature. Furthermore, it is well known that meridians on $S$ are geodesics.
$If \gamma$ is a convex curve, then we obtain a convex surface of revolution having as geodesics (i) meridians,(ii)  closed curves and (iii) geodesics dense in a ring $r\ge c$ (\cite[pp.~85-86]{Arnold:89}, \cite[Theorem~5.3,p.~354]{O'Neill:06}).

A fundamental property for the qualitative behaviour of geodesics on $S$ is given by Clairaut's theorem. There are two similar types of Clairaut's theorem mentioned at most classical books in differential geometry having a cosine relation or a sine relation, respectively.

Let $r(t)$ be the arc length parametrization of a geodesic $\sigma(t)$ on $S.$ We denote by $\rho(t_{0})$ the distance from a point $\sigma (t_{0})$ to the $z-$axis of rotation and by $\alpha(t_{0})$ the angle between $\sigma(t)$ and a parallel at the point of their intersection $\sigma(t_{0})$
The first type of Clairaut's theorem states that (\cite[Example~2.7.3,(2.96), p.~123]{VToponogov:05},\cite[p.~257]{Do Carmo:76},\cite[p.~137]{Pogorelov:59}):
\[\rho(t_{0})\cos\alpha(t_{0})=c, \] where $c$ is a constant number (Clairaut's cosine relation).

We denote by $\beta(t_{0})$ the angle between $\sigma(t)$ and a meridian at the point of their intersection $\sigma(t_{0}).$

The second type of Clairaut's theorem states that (\cite[p.~214]{Spivak:79},\cite[pp.~216-225]{Oprea:07},\cite[Theorem~5.3]{O'Neill:06},\cite[p.~134]{Stuik:61}):
\[\rho(t_{0})\sin\beta(t_{0})=c, \]
where $c$ is a constant number (Clairaut's sine relation).

We call Clairaut constant the constant number $c,$ which satisfies Clairaut's (cosine or sine) relation.

In the 17th century, Fermat stated the first problem on the theory of branching geodesics in $\mathbb{R}^{2}.$ The existence, uniqueness and properties of the Fermat solution on a surface $S$ in $\mathbb{R}^{3}$ are studied in \cite[Theorem~1.2, (1.2.1),(1.2.2),Definition,pp.~510-511]{Karcher:77}, \cite{IvanTuzh:092}, \cite{IvanovTuzhilin:01b}, \cite{Cots/Zach:11},\cite[Proposition~1,p.~45]{Zachos/Cots:10}, by using variational methods.

Let $A_{1},A_{2},A_{3}$ be three fixed non-collinear points that do not lie on the same geodesic on a $C^{2}$ surface $S$ in $\mathbb{R}^{3}$ and a positive real number $b_{i}$ corresponds to the point $A_{i}.$ We denote by $\sigma_{i0}$ the length of the geodesic arc $A_{i}A_{0}$ for $i=1,2,3.$

\begin{problem}[Weighted Fermat problem for three non-collinear points on $S$ {\cite{Cots/Zach:11},\cite[Problem~1,p.~45]{Zachos/Cots:10}}]
Find a point $A_{0}$ (weighted Fermat-Torricelli point), such that
\begin{equation} \label{eq:001}
f(A_{0})=\sum_{i=1}^{3}b_{i}\sigma_{0i}\to min.
\end{equation}
\end{problem}

\begin{problem}[The inverse weighted Fermat-Torricelli problem on $S$ {\cite{Cots/Zach:11},\cite[Problem~2,p.~52]{Zachos/Cots:10}}]
Given a point $A_{0}\notin \{A_{1},A_{2},A_{3}\}$ on $S$
does there exist a unique set of positive real numbers (weights) $b_{i0},$ such
that:
\[b_{10}+b_{20}+b_{30}=1>0,\]
\[f(A_{0})=\sum_{i=1}^{3}b_{i0}\sigma_{i0} \to min.\]

\end{problem}
The existence and uniqueness of the weighted Fermat-Torricelli point is proved in \cite{Cots/Zach:11},\cite[Proposition~1,pp.~45]{Zachos/Cots:10}.

A positive answer for the inverse weighted Fermat-Torricelli problem on $S$ is given in \cite{Cots/Zach:11},\cite[Proposition~5,p.~52]{Zachos/Cots:10}.
We call Fermat constants the unique triad of weights $\{b_{10},b_{20},b_{30}\},$ which give the unique solution for the inverse weighted Fermat-Torricelli problem on $S.$
The three branching geodesics $A_{1}A_{0},$ $A_{2}A_{0},$ $A_{3}A_{0},$ which meet at the unique weighted Fermat-Torricelli point $A_{0}$ form a unique weighted geodesic (Fermat-Torricelli) tree on $S.$

In this paper, we show that the only case that Fermat constants completely determine Clairaut constants for three branching geodesics that meet at a point on a surface of revolution is the standard sphere in $\mathbb{R}^{3}.$

\section{Fermat Constants completely determine Clairaut constants for branching geodesics on a standard sphere}

In this section, we shall prove that the only case that Fermat constants completely determine Clairaut constants for three branching geodesics that meet at the weighted Fermat-Torricelli point is the standard sphere in $\mathbb{R}^{3}.$

Let $\{A_{1},A_{2},A_{3}\}$ be three fixed non-collinear points on $S,$ which do not lie on a geodesic and $A_{0}\notin \{A_{1},A_{2},A_{3}\}$ be a point on $S$. We denote by $b_{i}$ the weight, which corresponds to $A_{i},$ by $\phi_{i0j}$ the angle between the geodesic arcs $A_{i}A_{0}$ and $A_{j}A_{0},$ and by $\vec{U}_{A_{i}A_{j}}$ the unit tangent vector of the geodesic arc $A_{i}A_{j}$ at $A_{i}$ for $i,j=1,2,3.$ We assume that there is a positive number $I,$ such that the injectivity radius of $S$ $inj(S):=I.$ Thus, we may consider a bounded region $D$ on $S$, which is a subset of a disk $B_{X;I}$ with center $X$ and radius $R=I,$ such that $\{A_{0,}A_{1},A_{2},A_{3}\}\in D.$

We need the following three lemmas, which also hold on a $C^{2}$ surface in $\mathbb{R}^{3},$ in order to prove the two main results (Theorems~\ref{mainres1},~\ref{mainres2}).

\begin{lemma}[Characterization of the weighted Fermat-Torricelli solution {\cite{Cots/Zach:11},\cite[Proposition~6, p.~53]{Zachos/Cots:10}}] \label{lem1}
(I) and (II) conditions are equivalent:

(I) All the following inequalities are satisfied simultaneously:
\[
\left\|
b_{2}\vec{U}_{A_{1}A_{2}}+b_{3}\vec{U}_{A_{1}A_{3}}\right\|>
b_{1},
\]
\[
\left\|
b_{1}\vec{U}_{A_{2}A_{1}}+b_{3}\vec{U}_{A_{2}A_{3}}\right\|>
b_{2},
\]

\[
\left\|
b_{1}\vec{U}_{A_{3}A_{1}}+b_{2}\vec{U}_{A_{3}A_{2}}\right\|>
b_{3},
\]
(II) The point $A_{0}\notin \{A_{1},A_{2},A_{3}\}$
(weighted Fermat-Torricelli point) and does not
belong to the geodesic arcs $A_{1}A_{2},$ $A_{2}A_{3}$
and $A_{1}A_{3}.$
\end{lemma}

Lemma~\ref{lem1} gives us the ability to select a triad of weights (Fermat constants) $\{b_{1},b_{2},b_{3}\},$ such that the weighted Fermat-Torricelli point $A_{0}\notin \{A_{1},A_{2},A_{3}\}.$

\begin{lemma}[Unique solution for the weighted Fermat-Torricelli problem on $S$ {\cite{Cots/Zach:11},\cite[Proposition~2, p.~48]{Zachos/Cots:10}}]\label{lem2}
If  $A_{0}\notin \{A_{1},A_{2},A_{3}\},$ then each angle $\phi_{i0j}$ is expressed as a
function of $b_{1},$ $b_{2}$ and $b_{3}:$
\begin{equation}\label{cosi0j}
\phi_{i0j}=\arccos\left(\frac{b_{k}^2-b_{i}^2-b_{j}^2}{2b_{i}b_{j}}\right)
\end{equation}
for $i,j,k=1,2,3,$ and $k\ne i\ne j.$
\end{lemma}

\begin{lemma}[Unique solution for the inverse weighted Fermat-Torricelli problem on $S$ {\cite{Cots/Zach:11},\cite[Proposition~5, p.~52]{Zachos/Cots:10}}]\label{lem3}
The weight $b_{i}$ are uniquely determined by the formula:
\begin{equation}\label{inverse111}
b_{i}=\frac{C}{1+\frac{\sin{\phi_{i0k}}}{\sin{\phi_{j0k}}}+\frac{\sin{\phi_{i0j}}}{\sin{\phi_{j0k}}}},
\end{equation}
where
\begin{equation}\label{ratioji2}
\frac{b_{j}}{b_{i}}=\frac{\sin{\phi_{i0k}}}{\sin{\phi_{j0k}}}.
\end{equation}
for $i,j,k=1,2,3.$
\end{lemma}

We consider a coordinate system $A_{0}XY$ having as axes a parallel $XX^{\prime}$ ($X$ may coincide with $X^{\prime}$) and a meridian $YY^{\prime}$ on $S,$ which intersect at $A_{0}$ orthogonally. Hence, we may consider $A_{0}$ as the origin of the system. We note that this coordinate system divides $S$ into four quadrants, counting each quadrant in the usual way like in $\mathbb{R}^{2}.$ Without loss of generality, we assume that $A_{1}$ belongs to the second quadrant, $A_{2}$ to the fourth quadrant and $A_{3}$ to the third quadrant. Hence, the symmetrical points $A_{1}^{\prime},$ $A_{2}^{\prime},$ $A_{3}^{\prime},$ with respect to $A_{0}$ belong to the fourth, second and first quadrant, respectively, such that $A_{i}A_{0}A_{i}^{\prime}$ form a unique geodesic arc on $S,$ and $A_{i},A_{0},A_{i}^{\prime}\in D,$ for $i=1,2,3.$

Furthermore, we assume that $\frac{\pi}{2}<\phi_{102},\phi_{203},\phi_{301}<\pi$ and we consider as positive orientation for the angle counting clockwise, with respect to $A_{0}.$
We set $\alpha_{1}\equiv \angle A_{1}A_{0}X,$ $\alpha_{2}\equiv \angle A_{2}A_{0}X^{\prime},$ $\alpha_{3}\equiv \angle A_{3}A_{0}X^{\prime},$
$\beta_{1}\equiv \angle A_{1}A_{0}Y,$ $\beta_{2}\equiv \angle A_{2}A_{0}Y^{\prime},$ $\beta_{3}\equiv \angle A_{3}A_{0}Y.$

Hence, we obtain three branching geodesics $\{A_{1}A_{0},A_{2}A_{0},A_{3}A_{0}\}\in D$, which intersect at $A_{0}.$

We set $w(x_,y_,z)\equiv \frac{x^2+y^2-z^2}{2xy}.$

\begin{theorem}\label{mainres1} The Clairaut constants $\{c_{1},c_{2},c_{3}\}$ depend on the Fermat constants $\{b_{1},b_{2},b_{3}\}$ and an angle formed by a geodesic arc and a parallel on $S,$ except the standard sphere $S_{R}$ in $\mathbb{R}^{3}.$
\end{theorem}

\begin{proof}
By applying Clairaut's cosine relation for each geodesic arc $A_{i}A_{0},$ for $i=1,2,3,$ we get:

\begin{equation}\label{clairautconstants}
\frac{c_{1}}{\cos\alpha_{1}}=\frac{c_{2}}{\cos\alpha_{2}}=\frac{c_{3}}{\cos\alpha_{3}}=\rho(t_{0}).
\end{equation}

where $\alpha_{1}=\beta_{1}+\frac{\pi}{2},$ $\alpha_{2}=\beta_{2}+\frac{\pi}{2},$ $\alpha_{3}=\beta_{3}-\frac{\pi}{2}.$

Therefore, by taking into account that $\frac{\pi}{2}<\alpha_{1}, \alpha_{2}<\pi$ and $0<\alpha_{3}<\frac{\pi}{2},$
(\ref{clairautconstants}) yields:
$c_{1},c_{2}<0$ and $c_{3}>0.$

Hence, we obtain the following angular relations:

\begin{equation}\label{angularrelation1}
\phi_{102}=\pi+\alpha_{1}-\alpha_{2},
\end{equation}

\begin{equation}\label{angularrelation2}
\phi_{301}=\pi+\alpha_{3}-\alpha_{1}
\end{equation}

\begin{equation}\label{angularrelation3}
\phi_{203}=\alpha_{2}-\alpha_{3},
\end{equation}

such that:
\[\phi_{102}+\phi_{203}+\phi_{301}=2\pi.\]

Taking into account (\ref{angularrelation1}) and (\ref{clairautconstants}), we derive a quadratic equation with respect to $\frac{c_{2}}{c_{1}},$
which gives:
\begin{equation}\label{ratio21}
0<\frac{c_{2}}{c_{1}}=w(b_{1},b_{2},b_{3})-|\tan\alpha_{1}|\sqrt{1-w(b_{1},b_{2},b_{3})^2}<\frac{w(b_{1},b_{2},b_{3})}{\cos^2\alpha_{1}}.
\end{equation}

Taking into account (\ref{angularrelation2}) and (\ref{clairautconstants}), we derive a quadratic equation with respect to $\frac{c_{3}}{c_{1}},$
which gives:
\begin{equation}\label{ratio31}
0>\frac{c_{3}}{c_{1}}=w(b_{1},b_{3},b_{2})-|\tan\alpha_{1}|\sqrt{1-w(b_{1},b_{3},b_{2})^2}.
\end{equation}

Hence, (\ref{ratio21}), (\ref{ratio31}) and (\ref{clairautconstants}) yield:

\[c_{1}=\rho(t_{0}) \cos\alpha_{1},\]
\[c_{2}=\rho(t_{0})\cos\alpha_{1}(w(b_{1},b_{2},b_{3})-|\tan\alpha_{1}|\sqrt{1-w(b_{1},b_{2},b_{3})^2}),\]
\[c_{3}=\rho(t_{0})\cos\alpha_{1}(w(b_{1},b_{3},b_{2})-|\tan\alpha_{1}|\sqrt{1-w(b_{1},b_{3},b_{2})^2}).\]

\end{proof}

\begin{remark}
Suppose that we rotate by an angle $\phi$ the triad of geodesic arcs $\{A_{1}A_{0},A_{2}A_{0},A_{3}A_{0}\}$ (weighted geodesic tree) with respect to $A_{0}$ on $S$ having fixed angles $\phi_{102},\phi_{203}, \phi_{301},$ such that the new position $\{A_{1}^{\prime}A_{0},A_{2}^{\prime}A_{0},A_{3}^{\prime}A_{0}\} \in D.$
We note that as a consequence of Theorem~\ref{mainres1} the corresponding Fermat constants $\{b_{10},b_{20},b_{30}\}$ remain the same, but Clairaut constants change, because they depend on the value of $\alpha_{1}.$
\end{remark}

\begin{theorem}\label{mainres2}
The Clairaut constants $\{c_{1},c_{2},c_{3}\}$ are completely determined by the Fermat constants $\{b_{1},b_{2},b_{3}\}:$  $\frac{c_{i}}{\sum_{i=1}^{3}c_{i}}\equiv b_{i}$  for $i=1,2,3,$ on the standard sphere $S_{R}$ in $\mathbb{R}^{3}.$
\end{theorem}

\begin{proof}
By taking into account (\ref{ratioji2}) from Lemma~\ref{lem3} and the Clairaut's sine relation  for three branching geodesics $\{A_{1}A_{0},A_{2}A_{0},A_{3}A_{0}\}$ (meridians), which meet at the weighted Fermat-Torricelli point $A_{0}$ on  $S_{R},$
we get:

\begin{equation}\label{clairautconstantssin}
\frac{c_{1}}{\sum_{i=1}^{3}c_{i}\sin\phi_{203}}=\frac{c_{2}}{\sum_{i=1}^{3}c_{i}\sin\phi_{301}}=\frac{c_{3}}{\sum_{i=1}^{3}c_{i}\sin\phi_{102}}=\frac{\rho(t_{0})}{\sum_{i=1}^{3}c_{i}}
\end{equation}

and

\begin{equation}\label{Fermatconstantssin}
\frac{b_{1}}{\sin\phi_{203}}=\frac{b_{2}}{\sin\phi_{301}}=\frac{b_{3}}{\sin\phi_{102}}=\frac{2b_{1}b_{2}b_{3}}{\sqrt{(b_{1}+b_{2}-b_{3})(b_{2}+b_{3}-b_{1})(b_{3}+b_{2}-b_{1})}},
\end{equation}

such that $c_{i},b_{i}>0,$ for $i=1,2,3.$

By dividing by parts (\ref{clairautconstantssin}) and (\ref{Fermatconstantssin}), respectively, we get

$\frac{c_{i}}{\sum_{i=1}^{3}c_{i}}\equiv b_{i}$  for $i=1,2,3.$

\end{proof}

\section{Concluding Remarks}

Manifolds all of whose geodesics are closed except spheres are studied in \cite{Besse:78} and \cite[11.10,pp.~422-425]{BergerGostiaux:88}.
It would be of interest to derive a relation between Clairaut constants and Fermat constants on manifolds all of whose geodesics are closed, except the standard sphere.

\end{document}